\documentclass{amsart}

\usepackage{amsmath,amssymb,amsthm,mathrsfs}
\usepackage{dsfont}
\usepackage[usenames,dvipsnames]{xcolor}
\usepackage{booktabs}
\usepackage{hyperref}
\usepackage{enumitem}

\makeatletter 
\newcommand\semiHuge{\@setfontsize\semiHuge{22.72}{27.38}}
\newcommand\semihuge{\@setfontsize\semihuge{18.93}{23.45}}
\newcommand\semiLARGE{\@setfontsize\semiLARGE{15.77}{19.90}}
\newcommand\semiLarge{\@setfontsize\semiLarge{13.15}{15.87}}
\newcommand\semilarge{\@setfontsize\semilarge{11.46}{13.80}}
\newcommand\seminormal{\@setfontsize\seminormal{10.46}{12.77}}
\makeatother

\newtheorem{thm}{Theorem}[section]
\newtheorem{cor}[thm]{Corollary}
\newtheorem{lem}[thm]{Lemma}

\theoremstyle{definition}

\newtheorem{ex}[thm]{Example}

\theoremstyle{remark}
\newtheorem{rem}[thm]{Remark}

\begin{document}

\title{Vector-valued concentration inequalities on the biased discrete cube}
\author{Miriam Gordin}
\date{}

\begin{abstract}
  We present vector-valued concentration inequalities for the biased measure on the discrete cube $\{-1,1\}^n$ with an optimal dependence on the bias parameter and the Rademacher type of the target Banach space. These results allow us to obtain novel vector-valued concentration inequalities for the measure given by a product of Poisson distributions. We further obtain lower bounds on the average distortion with respect to the biased measure of embeddings of the hypercube into Banach spaces of nontrivial type which imply average non-embeddability.
\end{abstract}

\keywords{Hamming cube, Banach spaces, Rademacher type, metric embeddings, functional inequalities.}

\subjclass{Primary: 60E15; Secondary: 46B85, 46B09, 30L15.}

\maketitle

\section{Introduction}

Let $(X,\|\cdot\|)$ be a Banach space and $\{-1,1\}^n$ the $n$-dimensional discrete cube.
The concentration of functions $f:\{-1,1\}^n \to (X,\|\cdot\|)$ with respect to the uniform measure on the discrete cube has been studied extensively in the past decades. Functional inequalities, such as Poincar\'e and log-Sobolev inequalities \cite{bobkov1997poincare,diaconis1996logarithmic,talagrand1993isoperimetry}, as well as Talagrand's influence inequality \cite{talagrand1994on}, capture this concentration for scalar-valued functions. Recent works \cite{CorderoErausquin2023,CorderoErausquin2024}  extend these classical results to functions that take values in Banach spaces satisfying certain conditions.

 The work of Pisier \cite{pisier1986} featured the following elegant inequality in this setting, which holds for any Banach space $(X, \|\cdot\|)$. For a function $f:\{-1,1\}^n \to X$, and $\zeta,\delta$ independent random vectors uniformly distributed on the discrete hypercube $\{-1,1\}^n$, and $p \geq 1$,
\begin{equation}\label{eq:pisier_cube_original}
\mathbb E \|f(\zeta) - \mathbb E f(\zeta)\|^p \leq C(n)^p \mathbb E \! \left\|\sum_{j=1}^n \delta_j D_jf(\zeta)\right\|^p.
\end{equation}
In fact, Pisier proved an inequality of the form \eqref{eq:pisier_cube_original} for the case where $\zeta,\delta$ are independent standard Gaussians, where $C(n)$ is independent of $n$.
However, in the case where $\zeta,\delta$ were random signs, he was only able to obtain a dimensional constant $C(n) \sim \log n$; Talagrand showed in \cite{talagrand1993isoperimetry} that  this is sharp.
If \eqref{eq:pisier_cube_original} were true with a dimension-free constant, this would settle a classical conjecture of Enflo  in  Banach space geometry concerning the equivalence of Enflo and Rademacher type \cite{enflo1978infinite}.

After more than 40 years, the conjecture was settled in the affirmative in \cite{Ivanisvili2020}. The key theorem was a novel form of vector-valued functional inequality resulting from a Markov semigroup interpolation argument, which yielded a dimension-free constant, settling Enflo's conjecture. 
More precisely, they show, in the notation of this paper, that for any function $f:\{-1,1\}^n \to X$,
$\varepsilon \sim \text{Unif}\left(\{-1,1\}^n\right)$, and $p \geq 1$,
  \begin{equation}\label{eq:IvHV}
   \left(\mathbb E \|f(\varepsilon) - \mathbb E f(\varepsilon)\|^p \right)^{\frac{1}{p}}\leq   \int_0^\infty \left(\mathbb E \! \left\|\sum_{j=1}^n \delta_j(t) D_jf(\varepsilon)\right\|^p \right)^{\frac{1}{p}} \, dt
    \end{equation}
    where 
    \[
    \delta_j(t) := \frac{\xi_j(t) - e^{-t}}{e^t - e^{-t}}
    \]
    and $\{\xi_j(t)\}_{j=1}^n$ are i.i.d.\ biased Rademacher random variables, independent of $\varepsilon$, with 
    \[
    \mathbb P(\xi_j(t) = 1 ) = \frac{1 + e^{-t}}{2} \quad \text{ and } \mathbb P(\xi_j(t) = -1 ) = \frac{1 - e^{-t}}{2}.
    \]
Recall that a Banach space $(X,\|\cdot\|)$ is said to have  (Rademacher) type $p \in [1,2]$ if there exists a $C \in (0,\infty)$ such that for every $n$ and $x_1, \dots, x_n \in X$, 
\[
\mathbb E \left\| \sum_{i=1}^n \zeta_ix_i\right\|^p \leq C^p  \sum_{i=1}^n \left\|  x_i\right\|^p,
\]
where the expectation is with respect to i.i.d.\ uniform Rademacher variables $\zeta_1, \dots, \zeta_n \sim \text{Unif}\left(\{-1,1\}\right)$.
Denote by $T_p(X)$  the optimal constant in the inequality.
In particular, for a Rademacher space of type $p \in [1,2]$, \eqref{eq:IvHV} implies that for  $f:\{-1,1\}^n \to X$, and $\varepsilon \sim \text{Unif}\left(\{-1,1\}^n\right)$, 
    \begin{equation} \label{eq:Lp_Poincare_cube_unif}
     \mathbb E \|f(\varepsilon) - \mathbb E f(\varepsilon)\|^p \leq C^p  \sum_{j=1}^n \mathbb E \! \left\| D_jf(\varepsilon)\right\|^p,
      \end{equation}
where $C$ does not depend on $n$. In contrast, the original inequality of Pisier \eqref{eq:pisier_cube_original} yields $C \sim \log n$. Thus, the approach of \cite{Ivanisvili2020} introducing appropriately biased random coefficients in the inequality of Pisier yields a functional inequality \eqref{eq:IvHV} that implies a vector-valued Poincar\'e inequality \eqref{eq:Lp_Poincare_cube_unif} with dimension-free constant.

\subsection{Biased measure on the discrete cube}

Let the \textit{biased} product measure on  the $n$-dimensional discrete cube $\{-1,1\}^n$ with  parameter $\alpha \in (0,1)$ be given by
\[
\mu = \bigotimes_{i=1}^n \mu_i
\]
where 
\[
\mu_i(+1) = \alpha \qquad \text{and} \qquad  \mu_i(-1) = 1-\alpha.
\]

In this paper, we present a vector-valued concentration inequality  on the discrete cube with respect to this \textit{biased} measure.
For the main result, we make no assumption on the target Banach space $(X, \|\cdot\|)$ in bounding the concentration of a function $f:\{-1,1\}^n \to (X, \|\cdot\|)$ with respect to $\mu$ .

We obtain an optimal dependence on the bias parameter and the Rademacher type of the Banach space. This result is an extension of the  inequality for the uniform measure on the hypercube in \cite{Ivanisvili2020}. If one takes $X = \mathbb R$ endowed with the Euclidean norm, we capture the same dependence on the bias as the scalar Poincar\'e inequality. Thus, in terms of the dependence of the bias, one loses nothing in going from the scalar to the vector-valued Poincar\'e inequality. The constant also improves the scalar-valued $L^p$ Poincar\'e inequality on the biased hypercube in \cite{talagrand1993isoperimetry}.

Our inequality on the biased hypercube allows us to prove a novel vector-valued concentration inequality on the space $\mathbb N^n$ endowed with the measure given by a product of Poisson(1) distributions. The optimal dependence on the bias parameter $\alpha$ is essential to our proof technique, which obtains the inequality as a result of an appropriate scaling limit of the hypercube inequality. In fact, by the central limit theorem, Pisier's Gaussian inequality can be derived from the inequality of \cite{Ivanisvili2020}. Thus, obtaining optimal constants  for vector-valued inequalities on the hypercube can yield important results on other spaces through appropriate scaling limits.

In \cite{eskenazis2023some}, Eskenazis obtains a different vector-valued  inequality on the biased hypercube in the form of an $L^p$ Poincar\' e inequality. 
However, the explicit dependence on the bias parameter $\alpha$ in that result is not sufficient in order to obtain the scaling limit mentioned above; this dependence requires a key observation about the differential structure on the biased discrete cube which is described in  Remark \ref{rem:Di_alpha} and Lemma \ref{lem:dirichlet}. See Section \ref{subsec:comparison} for a more detailed comparison of the results of this paper with those of \cite{eskenazis2023some}.

An important motivation for proving such vector-valued functional inequalities is applications in metric geometry. The result of \cite{Ivanisvili2020}  implies optimal lower bounds for the distortion of bi-Lipschitz embeddings from the discrete hypercube into an arbitrary Banach space of Rademacher type $p$. In \cite{eskenazis2023some}, Eskenazis  refines this for \textit{finite-dimensional} target Banach spaces.

In fact, \cite{Ivanisvili2020} has deeper consequences than this for metric embeddings, since it implies that the non-embeddability phenomenon occurs not only in the worst case but also on average with respect to the uniform measure on the hypercube.
This result  poses the natural question of whether average non-embeddability holds for other measures on the hypercube such as the biased product measure. We show that non-embeddability occurs even for bias parameters $\alpha$ depending on $n$ decaying slower than $\frac{1}{n}$. 

\medskip
\noindent
\textbf{Acknowledgements.} I wish to thank Ramon van Handel for his guidance, discussion, and feedback on this work. I am grateful to Alexandros Eskenazis for helpful conversations and feedback. This work was partially supported by the NSF Graduate Research Fellowship under grant DGE-2039656 and NSF grant DMS-2347954.

\subsection{Biased heat semigroup on the discrete cube}
\label{sec:biased_Hypercube_preliminaries}

For a function $f:\{-1,1\}^n \to X$, we define the discrete derivative on the biased hypercube as 
\[
D_i^\alpha f(x) := f(x) - \mathbb E_i f(x),
\]
where $\mathbb E_i f(x) = \mathbb E f(x_1, \dots, x_{i-1},\varepsilon_i, x_{i+1}, \dots, x_n)$ where $\varepsilon_i \sim \mu_i$. 

\begin{rem}\label{rem:Di_alpha}
This definition is consistent with \cite{eskenazis2023some}. We can also see it as a mixture of two other notions of discrete derivative. Let 
\begin{align*}
  D_i f (x) &:= \frac{1}{2}(f(x_1, \dots, x_{i-1},x_i, x_{i+1}, \dots, x_n) - f(x_1, \dots, x_{i-1},-x_i, x_{i+1}, \dots, x_n)) \\[1em]
  \partial_i f (x) &:=  \frac{1}{2}(f(x_1, \dots, x_{i-1},1, x_{i+1}, \dots, x_n) - f(x_1, \dots, x_{i-1},-1, x_{i+1}, \dots, x_n)). 
\end{align*}
Then we can express
\begin{equation}\label{eq:Di_alpha_identity}
D_i^\alpha f (\varepsilon) = (1-2\alpha) \partial_i f (\varepsilon) + D_i f(\varepsilon) = (1- 2\alpha + \varepsilon_i) \partial_i f(\varepsilon).
\end{equation}
Notice that $\partial_i f(\varepsilon)$ does not depend on the value of $\varepsilon_i$, so by the independence of the coordinates of $\varepsilon$, we have that the random quantities $\partial_i f(\varepsilon)$ and $\varepsilon_i$ are in fact independent! This will be a key observation in our further analysis.
\end{rem}
The generator of the random walk on the hypercube with stationary measure $\mu$ is given by 
\[
\Delta = - \sum_{i=1}^n D_i^\alpha.
\]
Recall that the Dirichlet form associated to $\Delta$ and $\mu$  is given by 
\[
\mathcal E(f,g) := - \mathbb E_\mu[f(\varepsilon) \Delta g(\varepsilon)].
\]
The following representation of the Dirichlet form is critical to obtaining the optimal dependence on $\alpha$ in the constant in Theorem \ref{thm:biased_hypercube_pisier}.
\begin{lem}\label{lem:dirichlet}
  The Dirichlet form associated with $\Delta$ is given by 
  \begin{equation}
    \mathcal E(f,g) = 
    4\alpha (1- \alpha ) \sum_{i=1}^n \mathbb E_\mu [D_i f(\varepsilon) D_i g(\varepsilon)] \notag
  \end{equation}
\end{lem}
\begin{proof}
Notice that by the identity \eqref{eq:Di_alpha_identity},
\[
  \mathbb E_\mu[f(\varepsilon)D_i^\alpha g(\varepsilon)] = 
  \mathbb E_\mu[(1- 2\alpha + \varepsilon_i)f(\varepsilon)\partial_ig(\varepsilon)]
\]
Recalling that $\partial_ig(\varepsilon)$ is independent of $\varepsilon_i$, we compute that 
\begin{align*}
  \mathbb E_i [(1- 2\alpha + \varepsilon_i)f(\varepsilon)\partial_ig(\varepsilon)]
&= \alpha (2- 2\alpha )f(1)\partial_ig(\varepsilon)
+ (1- \alpha ) (-2 \alpha) f(-1)  \partial_i g(\varepsilon) \\
&= 4\alpha (1- \alpha )\partial_i f(\varepsilon)\partial_i g(\varepsilon).
\end{align*}
Finally, using  the fact that $\varepsilon_i\partial_if(\varepsilon) = D_i f(\varepsilon)$ and $\varepsilon_i^2 = 1$, we have that 
\begin{equation}
  - \mathbb E[f(\varepsilon) \Delta g(\varepsilon)] = 
  4\alpha (1- \alpha ) \sum_{i=1}^n \mathbb E [D_i f(\varepsilon) D_i g(\varepsilon)].
\end{equation}

\end{proof}

We will consider the heat semigroup on the biased hypercube given by $P_t := e^{t \Delta}$.
It will be useful to work with the corresponding continuous-time random walk on the biased hypercube denoted by $\{X(t)\}_{t \geq 0}$. 

For the purpose of explicit computations, we may construct the random walk for $n=1$ as follows: let $Z_0,Z_1, Z_2, \dots$ be an i.i.d.\ sequence of biased Rademacher variables such that $Z_i \sim \mu$.

Let $\{N_t\}$ be a Poisson process with rate 1 independent of the $Z_i$s. Then the continuous-time random walk on the 1-dimensional biased hypercube with initial condition $X(0) = Z_0$ is given by 
\[
  X(t) = 
  Z_{N_t}\text{ for } t\geq0.
\]
Then for $n >1$, the random walk is given simply by  $n$  one-dimensional random walks independently on each coordinate. 
It can be readily checked that 
for $t \geq 0$,
  \[
  P_t f(x) = \mathbb E [f(X(t)) \mid X(0) = x].
  \]

We denote the heat kernel of the semigroup  by $p_t(\cdot, \cdot )$  so that 
\[
  \mathbb P (X(t) = y \mid X(0) = x) =: p_t(x,y) = \prod_{i=1}^n p_t(x_i,y_i),
\]
where it will be clear from context whether we are considering the heat kernel on a single coordinate or on the $n$-dimensional discrete hypercube.

We can write the heat kernel explicitly as 
\[
  p_t(x_i,y_i) = \frac{1-e^{-t}}{2} (2\alpha-1)y_i +  \frac{1+e^{-t}x_iy_i}{2}.
\]

\begin{rem}[on notation]
  Unless otherwise specified, throughout the paper $\varepsilon$ will denote a $\mu$-distributed random vector and all probabilities and expectations will be taken on the probability space on which the stationary random walk $\{X(t)\}_{t \geq 0}$ is defined. Furthermore, we will take $X(0) = \varepsilon$. 
    
  \end{rem}

We define the random vector $\delta(t)$ by
\[
  \delta_i(t) :
=\frac{D_i \, p_t(\varepsilon,X(t))}{p_t(\varepsilon,X(t))},
\]
where $D_i$ is acting on the first coordinate of $p_t(\cdot, \cdot )$. We can explicitly write $\delta(t)$ as 
\[
\delta_i(t) 
=\frac{e^{-t}\varepsilon_iX_i(t)}{(1-e^{-t})(2\alpha-1)X_i(t) + 1 + e^{-t}\varepsilon_iX_i(t)}.
\]

\begin{rem}
  In this work, we focus on the biased product measure which has the same bias $\alpha$ for each coordinate $i = 1, \dots, n$. The methods of proof of the results on the biased cube extend without issue to considering a different bias $\alpha_i$ for each coordinate. For clarity of notation and due to the focus on obtaining scaling limits as $\alpha$ becomes small, we have not expressed the results in this generality.
\end{rem}

\section{Results}

\subsection{Pisier-like inequality on the biased hypercube}

The main result is the following Pisier-like inequality on the biased hypercube, which is a generalization of the inequality for uniform measure on the hypercube in \cite{Ivanisvili2020}. Setting the bias $\alpha = \frac{1}{2}$ in the following theorem retrieves exactly the result in Theorem 1.4 of  \cite{Ivanisvili2020}.

\begin{thm} \label{thm:biased_hypercube_pisier}
  For any Banach space $(X,\|\cdot\|)$, function $f:\{-1,1\}^n \to X$, and $p \geq 1$, we have
  \[
  \left(\mathbb E\|f(\varepsilon) - \mathbb E f(\varepsilon)\|^p\right)^{\frac{1}{p}}
  \leq
  4 \alpha (1 - \alpha) \int_0^\infty  
  \left(\mathbb E\!\left\|\sum_{i=1}^n\delta_i(t)D_i f (\varepsilon) \right\|^p\right)^{\frac{1}{p}} dt.
  \]
\end{thm}

We have precise  estimates on the coefficients $\delta(t)$. First we observe that 
the coefficients are independent and identically distributed across indices $i = 1, \dots, n$. Furthermore, for any $t > 0$ and 
for any $x_i \in \{-1,1\}$,
\[
\mathbb E \big[\, \delta_i(t) \mid X_i(0) = x_i\big] = 0.
\]
We have explicit bounds on the $L^p$ norms of the $\delta_i(t)$, which are  critical to obtain Corollary \ref{cor:biased_hypercube_type}.

\begin{lem} \label{lem:eta_unif_bound}
  For $\alpha < \frac{1}{2}$ and $p \geq 1$,
 \[
   \sup_{x \in \{-1,1\}^n} \mathbb E \Big[\big|\delta_i(t)\big|^p \mid X(0) = x\Big] \leq  (2\alpha )^{1-p} e^{-tp}(1 - e^{-t})^{1-p}.
\]
\end{lem}

Applying a standard symmetrization argument, we obtain the following vector-valued biased Poincar\'e inequality.

\begin{cor} \label{cor:biased_hypercube_type}
  For any Banach space $(X,\|\cdot\|)$ of Rademacher type $p \in [1,2]$, function $f:\{-1,1\}^n \to X$, and $\alpha < \frac{1}{2}$, we have that
  \begin{align*}
    \left(\mathbb E\|f(\varepsilon) - \mathbb E f(\varepsilon)\|^p\right)^{\frac{1}{p}}
    &\leq
     32 T_p(X)\alpha^\frac{1}{p} 
    \left(\sum_{i=1}^n\mathbb E\!
    \left\|D_i f(\varepsilon)\right\|^p\right)^{\frac{1}{p}}.
  \end{align*}
\end{cor}

\begin{rem}
  In order to showcase the optimal dependence on $\alpha$, we impose the assumption that $\alpha < \frac{1}{2}$. By symmetry, the case $\alpha \geq \frac{1}{2}$ can be obtained by replacing $\alpha$ by $1 - \alpha$. 
\end{rem}
Through a simple example, one can see that these inequalities are optimal in the sense of capturing the dependence on $\alpha$ and the type $p$ of $X$.

\begin{ex}
  Let $(X,\|\cdot\|) = (\mathbb R^n, \| \cdot \|_{\ell^p} )$ and
\[
f(\varepsilon) := n^{-\frac{1}{p}} \begin{bmatrix}
  \varepsilon_1 \\
  \varepsilon_2 \\
  \vdots \\
  \varepsilon_n \\
\end{bmatrix}.
\]
Suppose without loss of generality that $\alpha < \frac{1}{2}$.
Then by the law of large numbers,
\[
\|f(\varepsilon) - \mathbb E f(\varepsilon)\| = \left( \frac{1}{n}\sum_{i=1}^n \left|\varepsilon_i - \mathbb E \varepsilon_i \right|^p\right)^{\frac{1}{p}} \xrightarrow[n \to \infty]{ }  \left(\mathbb E |\varepsilon_1 - \mathbb E \varepsilon_1 |^p \right)^\frac{1}{p} \sim \alpha^\frac{1}{p}.
\]
Meanwhile, notice also that for this choice of $f$, we can bound the right-hand side of Corollary \ref{cor:biased_hypercube_type} by $\alpha^\frac{1}{p}$.
Thus, we have shown that Theorem \ref{thm:biased_hypercube_pisier} and Corollary \ref{cor:biased_hypercube_type} 
attain the optimal dependence in terms of $\alpha$, $p$, and $n$, up to universal constants. In fact, even the inequality
\[
  \mathbb E \|f(\varepsilon) - \mathbb E f(\varepsilon)\|
  \leq
  \int_0^\infty  4 \alpha (1 - \alpha)
  \left(\mathbb E\!\left\|\sum_{i=1}^n\delta_i(t)D_i f (\varepsilon) \right\|^p\right)^{\frac{1}{p}} dt
  \]
  which is weaker than that of Theorem \ref{thm:biased_hypercube_pisier}
  attains the optimal dependence on $\alpha$, $p$, and $n$.

\end{ex}

\subsection{Comparison with inequality obtained in \cite{eskenazis2023some}}
\label{subsec:comparison}
As noted in the introduction, in \cite{eskenazis2023some}, Eskenazis obtains a vector-valued Poincar\'e inequality on the biased hypercube. In the notation of this paper, Theorem 14 of \cite{eskenazis2023some} states that for any function $f:\{-1,1\}^n \to (X,\|\cdot\|)$ where $X$ is a Banach space of type $p \in [1,2]$, 
\[
  \left(\mathbb E\|f(\varepsilon) - \mathbb E f(\varepsilon)\|^p\right)^{\frac{1}{p}}
  \leq 2 \pi T_p(X)
  \left(\sum_{i=1}^n\mathbb E\!
  \left\|D_i^\alpha f(\varepsilon)\right\|^p\right)^{\frac{1}{p}}.
\]
Recall by the identity \eqref{eq:Di_alpha_identity} and the fact that $\partial_if(\varepsilon)$ is independent of $\varepsilon_i$ (see Remark \ref{rem:Di_alpha}),
\[
  \mathbb E\!
  \left\|D_i^\alpha f(\varepsilon)\right\|^p
  = \mathbb E\!\left[(1- 2\alpha + \varepsilon_i)^p\right]\mathbb E \left\|\partial_i f(\varepsilon)\right\|^p.
\]
As computed in the example above, when $\alpha < \frac{1}{2}$, 
\[
  \mathbb E\!\left[(1- 2\alpha + \varepsilon_i)^p\right] \sim \alpha.
\]
Thus, up to universal constants, Theorem 14 of \cite{eskenazis2023some} and Corollary \ref{cor:biased_hypercube_type} are equivalent, but identity \eqref{eq:Di_alpha_identity} is necessary in order to obtain the explicit dependence on the bias parameter.

As an intermediate step in the proof of Theorem 14 of \cite{eskenazis2023some}, Eskenazis also states a Pisier-like inequality similar to that of Theorem \ref{thm:biased_hypercube_pisier}; however, this result does not suffice in order to obtain the scaling limit to prove the Pisier inequality for the product of Poisson distributions present in Theorem \ref{thm:poisson_inequality}.

\subsection{Average non-embeddability of the biased discrete hypercube}
\label{sec:avg_embed}

The study of average distortion of metric embeddings was introduced by Rabinovich in \cite{rabinovich2008average} and has led to new insights  in the Ribe program \cite{naor2012introduction} and combinatorics \cite{lee2017separators}. Rabinovich considers average distortion for uniform measures on finite metric spaces -- we state the definition for \textit{any} probability measure, which is a natural extension.\footnote{When the metric space is \textit{infinite}, average distortion has been defined analogously for all probability measures, see e.g.\ \cite{naor2021average}.}

Let $(M,d)$ be any metric space and let $\nu$ be a measure on $M$.
Then we say an embedding $f:(M,d) \to (X, \|\cdot\|)$ has $\nu$-average distortion $D \in \mathbb R$ if $f$ is $D$-Lipschitz and
\[
\mathbb E_{\nu \otimes \nu} \|f(\varepsilon) - f(\varepsilon')\| \geq \mathbb E_{\nu \otimes \nu} |d(\varepsilon,\varepsilon')|.
\]  

Consider the discrete hypercube $\{-1,1\}^n$ with the Hamming metric which we will denote as  $d\, (\cdot,\cdot)$, where for $x,y \in\{-1,1\}^n $,
\[
d(x,y) = \sum_{i=1}^n \mathds{1}_{\{x_i \neq y_i\}}
\]
Corollary \ref{cor:biased_hypercube_type} implies the following lower bound on the $\mu$-average distortion for embeddings of the hypercube into a Banach space of type $p$. 

\begin{cor} \label{cor:avg_nonembed}
  Let $(X, \|\cdot\|)$ be a Banach space of Rademacher type $p$ and $\mu$ the $\alpha$-biased product measure on $\{-1,1\}^n$ as defined above.  Then for any $f: \{-1,1\}^n \to (X, \|\cdot\|)$ with $f$ having $\mu$-average distortion $D$,
  \[
  D \gtrsim_X ( \alpha n)^{1-\frac{1}{p}},
  \]
  where the result holds up to universal constants depending on $T_p(X)$.

\end{cor}

\begin{ex}
  Let 
  $
  \alpha_n
  $ be the bias parameter of the product measure $\mu$ as defined above such that  $\lim_{n\to \infty} n\alpha_n \to \infty$.  Suppose $f: \{-1,1\}^n \to (X, \|\cdot\|)$ has $\mu$-average distortion $D$. Then by Corollary \ref{cor:avg_nonembed} 
  \[
  D \gtrsim \left(n\alpha_n\right)^{1-\frac{1}{p}}  \to \infty.
  \]
   Thus, even for very sparse random vector distributions on the hypercube (in terms of the number of entries equal to 1), the distortion of $D$ must still grow with $n$ for any Banach space $X$ of nontrivial type, i.e.\ $p > 1$.

   In fact, this condition on $\alpha_n$ is optimal: take $(X,\|\cdot\|)$ to be $(\mathbb R^n, |\cdot |)$ and consider the identity map $f:\{-1,1\}^n \to \mathbb R^n$, which has Lipschitz constant 2. If $\alpha_n \sim \frac{1}{n}$, then 
\[
\mathbb E_{\mu \otimes \mu} |f(\varepsilon) - f(\varepsilon')| = 2 \mathbb E_{\mu \otimes \mu}[d(\varepsilon,\varepsilon')^\frac{1}{2}] \geq 2 \mathbb E_{\mu \otimes \mu}[\min\{d(\varepsilon,\varepsilon'),1\}].
%
\]
Notice that $\min\{d(\varepsilon,\varepsilon'),1\} = 0$ with probability $(1 - 2 \alpha_n(1-\alpha_n))^n$. Therefore, 
\[
\mathbb E_{\mu \otimes \mu} |f(\varepsilon) - f(\varepsilon')| \geq
2 \! \left(1 - (1 - 2 \alpha_n(1-\alpha_n))^n\right) \sim 1
\]
On the other hand, $\mathbb E[d(\varepsilon,\varepsilon')] = 2 n  \alpha_n (1  - \alpha_n) \sim 1$. Therefore, when $\alpha_n \lesssim \frac{1}{n}$, constant average distortion can be achieved. 
\end{ex}

\subsection{Pisier-like inequality for product Poisson measure}

Let $N = \{N_i\}_{i=1, \dots, m}$ be a vector of i.i.d.\ Poisson(1) random variables.

\begin{thm} \label{thm:poisson_inequality}
  For any Banach space $(X,\|\cdot\|)$, bounded function $f:\mathbb N^m \to X$, $1 \leq p < \infty$, and ,
  \[
\left(\mathbb E\|f(N) - \mathbb E f(N)\|^p\right)^{\frac{1}{p}} \leq
\int_0^\infty \!
\left(\mathbb E\!\left\|\sum_{i=1}^m \tilde\eta_i(t) D_i^{\mathbb{Z}} f(N)\right\|^p\right)^{\frac{1}{p}}  dt,
  \]
  where
  \[
D_i^{\mathbb{Z}} := f(x_1, \dots, x_i+1, \dots, x_m) - f(x_1, \dots, x_i, \dots, x_m)
  \]
  and
  \[
  \tilde\eta_i(t) := e^{-t}- \frac{e^{-t}}{(1-e^{-t})}\eta_i(t),
  \]
  for $\eta_i(t)$ independent $ \mathrm{Poisson}(1-e^{-t})$ random variables over all $i=1,\dots,m$, also independent of $N$.
\end{thm}

\begin{cor}\label{cor:poisson_type_ineq}
  For any Banach space $(X,\|\cdot\|)$ of Rademacher type $p \in [1,2]$ with optimal constant $T_p(X)$, bounded function $f:\mathbb N^m \to X$,
  \[
\left(\mathbb E\|f(N) - \mathbb E f(N)\|^p\right)^{\frac{1}{p}}\leq 4\,  T_p(X)
\left(\sum_{i=1}^m \mathbb E\!
\left\|D_i^{\mathbb{Z}} f(N) \right\|^p\right)^{\frac{1}{p}}.
  \]
\end{cor}

\section{Proofs}

The following lemma makes explicit the role of $\delta(t)$ as a discrete partial derivative of the semigroup.
We denote $\mathbb E_x[\,  \cdot \, ] := \mathbb E[ \,\,\cdot \, \mid X(0) = x]$ for any fixed $x \in \{-1,1\}^n$.

\begin{lem} \label{lem:Pt}
  
We have for $t>0$
\begin{align*}
  D_i P_t f(x) &= \mathbb E_x[\delta_i(t)f(X(t))].
\end{align*}

\end{lem}
\begin{proof}

We have that by definition 
\[
  D_i P_t f(x) = \frac{1}{2}(P_t f(x) - P_t f(x_1, \dots, x_{i-1},-x_i, x_{i+1}, \dots, x_n)).
\]
By rewriting $P_t$ in terms of the heat kernel, we obtain
\begin{align*}
  D_i P_t f(x) &= \frac{1}{2} \left( \sum_{y\in \{-1,1\}^n}\! \left(\prod_{j=1}^n p_t(x_j,y_j) - p_t(-x_i,y_i)\prod_{\substack{j=1, \\ j \neq i\phantom{,}}}^n p_t(x_j,y_j)\right) f(y)  \right) \\
  &= \frac{1}{2} \left( \sum_{y\in \{-1,1\}^n} \! \!\left(1  - \frac{p_t(-x_i,y_i)}{p_t(x_i,y_i)}\right)  \prod_{j=1}^n p_t(x_j,y_j) f(y)  \right) \\
  &= \mathbb E_x \!\left[ \frac{1}{2} \left(1  - \frac{p_t(-x_i,X_i(t))}{p_t(x_i,X_i(t))}\right) f(X(t))\right].
\end{align*}
Observe that 
\[
   \frac{1}{2} \left(1  - \frac{p_t(-x_i,X_i(t))}{p_t(x,X(t))}\right)  = \frac{D_i \, p_t(x,X(t))}{p_t(x,X(t))},
\]
which is precisely the definition of $\delta_i(t)$. 
\end{proof}

\subsection{Proofs of Theorem \ref{thm:biased_hypercube_pisier} and Corollary \ref{cor:biased_hypercube_type}}

The following proof follows closely the proof of Theorem 1.4 in \cite{Ivanisvili2020} except for the key step of using the Dirichlet form derived in Lemma \ref{lem:dirichlet}

\begin{proof}[Proof of Theorem \ref{thm:biased_hypercube_pisier}]

  By  Proposition 1.3.1 of \cite{hytonen2016analysis}
  \[
  \left(\mathbb E\|f(\varepsilon) - \mathbb E f(\varepsilon)\|^p\right)^{\frac{1}{p}} = \sup_{\mathbb E \|g(\varepsilon)\|^{p'} \leq 1} \mathbb E[\langle g(\varepsilon),f(\varepsilon) - \mathbb E f(\varepsilon)\rangle]
  \]
  where $\frac{1}{p} + \frac{1}{p'} =1$.

  By the fundamental theorem of calculus, and using the stationarity of the Markov semigroup $P_t$, which implies that $P_0f = f$ and $\lim_{t\to\infty}P_tf = \mathbb E f(\varepsilon)$, we obtain
  \begin{align*}
  \mathbb E[\langle g(\varepsilon),f(\varepsilon) - \mathbb E f(\varepsilon)\rangle]
  &= - \int_0^\infty \mathbb E\!\left[\langle g(\varepsilon),\frac{d}{dt}P_tf(\varepsilon)\rangle\right] \, dt\\
  &= - \int_0^\infty \mathbb E\left[\langle g(\varepsilon),\Delta P_tf(\varepsilon)\rangle\right] \, dt\\
  &= \int_0^\infty \sum_{i=1}^n 4 \alpha (1 - \alpha)\mathbb E\!\left[\langle D_i g(\varepsilon),D_i P_tf(\varepsilon)\rangle\right] \, dt,
  \end{align*}
  where in the second to last line we have used the key representation of the Dirichlet form derived in Lemma \ref{lem:dirichlet}.
  By Lemma \ref{lem:Pt}, we have that
  \begin{align*}
    \mathbb E\!\left[\langle D_i g(\varepsilon),D_i P_t f(\varepsilon)\rangle\right] 
    &=\mathbb E\!\left[\langle  g(X(t)),\delta_i(t)D_i f(\varepsilon)\rangle\right] .
  \end{align*}
Therefore,
\begin{align*}
\mathbb E[\langle g(\varepsilon),f(\varepsilon) - \mathbb E f(\varepsilon)\rangle]
&= \int_0^\infty  4 \alpha (1 - \alpha)\mathbb E\!\left[\langle  g(X(t)), \sum_{i=1}^n\delta_i(t)D_i f(\varepsilon)\rangle\right] dt \\
&\leq
\int_0^\infty  4 \alpha (1 - \alpha)\left(\mathbb E\| g(X(t))\|^{p'}\right)^{\frac{1}{p}'}\!
\left(\mathbb E\left\|\sum_{i=1}^n\delta_i(t)D_i f(\varepsilon)\right\|^p\right)^{\frac{1}{p}} dt.
\end{align*}
By the reversibility of $P_t$, $X(t)$ is equal in distribution to $\varepsilon$ so that we obtain:
\[
\left(\mathbb E\|f(\varepsilon) - \mathbb E f(\varepsilon)\|^p\right)^{\frac{1}{p}}
\leq
\int_0^\infty  4 \alpha (1 - \alpha)
\left(\mathbb E\!\left\|\sum_{i=1}^n\delta_i(t)D_i f(\varepsilon)\right\|^p\right)^{\frac{1}{p}} dt.
\]
\end{proof}

The proof of Corollary \ref{cor:biased_hypercube_type} requires an estimate on the small moments of $\delta(t)$ which is captured in Lemma \ref{lem:eta_unif_bound}. It is particularly important that we obtain an upper bound which is uniform over all initial conditions of the underlying Markov process. 
We postpone the proof of this lemma until after the proof of Corollary \ref{cor:biased_hypercube_type}.

\begin{proof}[Proof of Corollary \ref{cor:biased_hypercube_type}]
 We begin by performing a routine symmetrization argument on the expectation on the right hand side of Theorem \ref{thm:biased_hypercube_pisier}. 
  Recall that 
  \[
\mathbb E \big[\, \delta_i(t) \mid X(0) = \varepsilon\big] = 0.
\]

  Let  $\zeta_1, \dots, \zeta_n$ be i.i.d.\ \textit{unbiased} Rademacher random variables independent of all other randomness. Then by Lemma 7.3 of \cite{van2014probability}, applied conditionally on $\varepsilon$, we have that 
  \begin{align*}
    \mathbb E\!\left\|\sum_{i=1}^n\delta_i(t)D_i f (\varepsilon) \right\|^p 
    \leq 2^p \mathbb E\!\left\|\sum_{i=1}^n \zeta_i\delta_i(t)D_i f (\varepsilon) \right\|^p
  \end{align*}
Then by the definition of Rademacher type and our assumption on $(X, \|\cdot\|)$, applied conditionally on $\delta(t)$ and $\varepsilon$,
\[
  \mathbb E\!\left\|\sum_{i=1}^n\delta_i(t)D_i f (\varepsilon) \right\|^p  
  \leq T_p(X)^p2^p\sum_{i=1}^n\mathbb E\!
  \left\|\delta_i(t)D_i f (\varepsilon)\right\|^p.
\]
Now applying the uniform bound on $\mathbb E_x |\delta_i(t)|^p$ obtained in Lemma \ref{lem:eta_unif_bound}, conditionally on $\varepsilon$, we obtain
\[
\mathbb E\!
  \left\|\delta_i(t)D_i f (\varepsilon)\right\|^p \leq (2\alpha )^{1-p} e^{-tp}(1 - e^{-t})^{1-p}\mathbb E\!
  \left\|D_i f (\varepsilon)\right\|^p
\]
Now by Theorem \ref{thm:biased_hypercube_pisier}, we have that 
\begin{align*}
  \left(\mathbb E\|f(\varepsilon) - \mathbb E f(\varepsilon)\|^p\right)^{\frac{1}{p}}
  &\leq T_p(X) 8 \alpha (1 - \alpha) \left( \sum_{i=1}^n\mathbb E\!
  \left\|D_i f(\varepsilon)\right\|^p\right)^{\frac{1}{p}} \int_0^\infty  
   (2\alpha )^{\frac{1}{p} - 1} e^{-t}(1 - e^{-t})^{\frac{1}{p} - 1} \, dt .
\end{align*}
Evaluating the  integral by a change of variables, we have that 
\[
\int_0^\infty  (2\alpha )^{\frac{1}{p} - 1} e^{-t}(1 - e^{-t})^{\frac{1}{p} - 1} \, dt = (2\alpha )^{\frac{1}{p} - 1} \int_0^1 u^{\frac{1}{p} - 1} \, du = p(2\alpha )^{\frac{1}{p} - 1},
\]
which yields the final result up to universal constants, using the assumption that $\alpha < \frac{1}{2}$ and the fact that $p \in [1,2]$.
\end{proof}

\begin{proof}[Proof of Lemma \ref{lem:eta_unif_bound}]
  Fix any $x \in \{-1,1\}^n$.
By definition of $\delta(t)$,
\begin{align*}
  \mathbb E_x \Big[\big|\delta_i(t)\big|^p\,\Big] = 
 \sum_{y \in \{-1,1\}} \left|\frac{e^{-t}}{2p_t(x_i,y)}\right|^p p_t(x_i,y) .
\end{align*}
Explicitly, 
\[
  \mathbb E_x \Big[\big|\delta_i(t)\big|^p\,\Big] = \frac{e^{-tp}}{2} \big(|(1-2\alpha+x_i)e^{-t} + 2\alpha|^{1-p} +
  |(2\alpha-1-x_i)e^{-t} + 2-2\alpha|^{1-p}\big).
\]
Denote
\begin{align*}
f_1(x_i,\alpha,t) &:= (1-2\alpha+x_i)e^{-t} + 2\alpha \\
f_2(x_i,\alpha,t) &:= (2\alpha-1-x_i)e^{-t} + 2-2\alpha.
\end{align*}
Notice that by construction $f_1,f_2 \geq 0$, since they arise from the heat kernel $p_t$.
Through a case-by-case analysis based on the initial condition $x$, we lower bound the terms $f_1$ and $f_2$ in order to obtain an upper bound on $\mathbb E |\delta_i(t)|^p$.
\noindent\hspace{-0.5in}
\begin{enumerate}[wide=0pt,,label={\sc {Case} \Roman*:}]
 \item  $x_i = -1$. We
 have that 
\[
  f_2(-1,\alpha,t) = 2\alpha e^{-t} + 2-2\alpha.
\]
Therefore, using $\alpha < \frac{1}{2}$,
 \begin{align*}
  f_1(-1,\alpha,t)
 = 2\alpha (1 - e^{-t})  \leq 2\alpha e^{-t} + 2-2\alpha  = f_2(-1,\alpha,t)
 \end{align*}
  for all $t \geq 0$. Thus,
   \[
    \mathbb E \Big[\big|\delta_i(t)\big|^p\,\Big| \, X_i(0) = -1\Big] \leq e^{-tp} f_1(-1,\alpha,t)^{1-p} = e^{-tp}  (2\alpha)^{1-p}(1 - e^{-t})^{1-p},
   \]
   using that $2\max\{a,b\} \geq a + b$. 
   \item $x_i = + 1$.
   We have that
   \begin{align*}
    f_1(1,\alpha,t) &= (2-2\alpha)e^{-t} + 2\alpha  \geq 2\alpha (1 - e^{-t}).
   \end{align*}
   Also, using the assumption that $\alpha < \frac{1}{2}$, we have that 
   \begin{align*}
    f_2(1,\alpha,t) &= (2\alpha-2)e^{-t} + 2-2\alpha = (2-2\alpha)(1-e^{-t}) \geq 2\alpha (1 - e^{-t}).
   \end{align*}

\end{enumerate}
Combining the analysis of the two cases, we obtain that 
\begin{align*}
  \mathbb E \Big[\big|\delta_i(t)\big|^p\,\Big| \, X_i(0) = x_i\Big] \leq e^{-tp} (2\alpha )^{1-p}(1 - e^{-t})^{1-p}
\end{align*}
uniformly in the value of $x$.
\end{proof}

\subsection{Proof of Corollary \ref{cor:avg_nonembed}}

\begin{proof}[Proof of Corollary \ref{cor:avg_nonembed}]
  Observe that since $f$ is $D$-Lipschitz, we have that 
  \[
  \|D_if\| \leq D.
  \]
  Therefore, we have by Corollary \ref{cor:biased_hypercube_type} that 
  \[
  T_p(X)(\alpha n)^\frac{1}{p} D \gtrsim  \left(\mathbb E\|f(\varepsilon) - \mathbb E f(\varepsilon)\|^p\right)^{\frac{1}{p}}.
  \]
  By Jensen's inequality,
  \[
    \left(\mathbb E\|f(\varepsilon) - \mathbb E f(\varepsilon)\|^p\right)^{\frac{1}{p}} \geq  \mathbb E\|f(\varepsilon) - \mathbb E f(\varepsilon)\|.
  \]
  Let $\varepsilon'$ be an independent identically distributed copy of $\varepsilon$.
Notice that 
  \[
    \mathbb E\|f(\varepsilon) - \mathbb E f(\varepsilon)\| = \frac{1}{2} \mathbb E\|f(\varepsilon) - \mathbb E f(\varepsilon)\|
    +    \frac{1}{2} \mathbb E\|f(\varepsilon') - \mathbb E f(\varepsilon')\|
    \geq \frac{1}{2}\mathbb E\|f(\varepsilon) -  f(\varepsilon')\|.
  \]
  By the assumption that $f$ has $\mu$-average distortion $D$, we have that 
  \[
    \mathbb E\|f(\varepsilon) -  f(\varepsilon')\| \geq \mathbb E |d(\varepsilon,\varepsilon')| \sim \alpha n,
  \]
  which concludes the proof.
\end{proof}

\subsection{Proofs of Theorem \ref{thm:poisson_inequality} and Corollary \ref{cor:poisson_type_ineq}}

The main idea of this section is to obtain the concentration inequality for Poisson random variables as the limit of appropriately scaled biased Rademachers. This would not be possible without the optimal dependence on $\alpha$ obtained in Theorem \ref{thm:biased_hypercube_pisier}. 
We will work on the $m \times n$-dimensional discrete hypercube $\{-1,1\}^{m \times n}$, where we fix $m$ and send $n$ to infinity. The following definitions are analogous to those given in Section \ref{sec:biased_Hypercube_preliminaries}, but with more convenient indexing for this setting.

We  consider the biased product measure on $\{-1,1\}^{m \times n}$ with parameter $\alpha  = \frac{1}{n}$,
\[
\mu = \bigotimes_{i=1}^m \bigotimes_{j=1}^n \mu_{ij}
\]
where 
\[
\mu_{ij}(+1) = \frac{1}{n} \qquad \text{and} \qquad  \mu_{ij}(-1) = 1-\frac{1}{n}.
\]

We will consider the heat semigroup $P_t$ associated with the random walk $X(t) = \{X_{ij}(t)\}_{i = 1, \dots, m; \, j = 1, \dots, n}$ on the discrete cube with bias $\alpha = \frac{1}{n}$. 

We denote
\begin{align*}
D_{ij} f (x) &:= \frac{1}{2}(f(x) - f(x_{11}, \dots, x_{i-1 \, j},-x_{ij},\, x_{i+1 \, j}\dots, x_{mn})),
\end{align*}
where we take the difference with respect to a single coordinate $x_{ij}$. 
Let $\varepsilon$ be a $\mu$-distributed random vector and denote as $Y \in \{0,1\}^{m \times n}$ the rescaled version of $\varepsilon$ given by
\[
Y_{ij} := \frac{\varepsilon_{ij} + 1}{2} \in \{0,1\}.
\]
We will apply our results on the biased hypercube to a ``structured'' function $g$ of the form $g:\{-1,1\}^{m\times n} \to X$ given by 
 \begin{equation}\label{eq:g_form}
   g(\varepsilon_{11}, \dots, \varepsilon_{1n}, \dots,\varepsilon_{mn})
   = f\!\left(\sum_{j=1}^n Y_{1j}, \sum_{j=1}^n Y_{2j}, \, \dots,  \sum_{j=1}^n Y_{mj} \right),
 \end{equation}
 for some function $f:\mathbb N^m \to X$.
 The following observation will be useful.

\begin{lem} \label{lem:}
  The discrete partial derivative of $g$ can be written as 
  \begin{align*}
    D_{ij}g(\varepsilon)
    &= \frac{1}{2}\left(
    f\!\left(\dots, \sum_{j'=1}^n Y_{ij'}, \, \dots \right) -
    f\!\left(\dots,  - \varepsilon_{ij}+ \sum_{j'=1}^n Y_{ij'}, \, \dots  \right) \right).
  \end{align*}
\end{lem}

\begin{proof}
  Using the definition of $g$ given in \eqref{eq:g_form}, we have that
  \begin{align*}
    D_{ij}g(\varepsilon)
    &= \frac{1}{2}\left(
    f\!\left(\dots, \sum_{j'=1}^n Y_{ij'}, \, \dots \right) -
    f\!\left(\dots, \frac{-\varepsilon_{ij} + 1}{2} + \sum_{\substack{j'=1 \\ j' \neq j}}^n Y_{ij'}, \, \dots  \right) \right).
  \end{align*}
  Noticing that 
  \[
  \frac{-\varepsilon_{ij} + 1}{2} + \sum_{\substack{j'=1 \\ j' \neq j}}^n Y_{ij'} = \frac{-\varepsilon_{ij} + 1}{2} + \varepsilon_{ij} - \varepsilon_{ij}+ \sum_{\substack{j'=1 \\ j' \neq j}}^n Y_{ij'} = - \varepsilon_{ij}+ \sum_{j'=1}^n Y_{ij'}
  \]
  completes the proof.
\end{proof}

We will denote the heat kernel of $P_t$ as $p^{(n)}_t( \cdot , \cdot )$ to emphasize the dependence on $n$ in the bias parameter. Explicitly, for any $x_{ij},y_{ij} \in \{-1,1\}$,
\begin{align*}
  p^{(n)}_t(x_{ij},y_{ij}) &= \frac{1-e^{-t}}{2} \left(\frac2n-1\right)y_{ij} +  \frac{1+e^{-t}y_{ij}x_{ij}}{2}.
\end{align*}
Let us denote 
\[
\delta_{ij}(t) :=  \frac{e^{-t}X_{ij}(t)\varepsilon_{ij}}{(1-e^{-t})(\frac{2}{n}-1)X_{ij}(t)+1 + e^{-t}\varepsilon_{ij}X_{ij}(t)}.
\]
where $ \varepsilon = X(0) \in \{-1,1\}^{m\times n}$.

\begin{lem}\label{lem:binomial_limit}
  Conditional on the event that $\sum_{j'=1}^n Y_{ij'} = k$, define independent binomial random variables 
  \[
  B^{(n)}_i \sim \mathrm{Binomial}\left(n-k,p^{(n)}_t(-1,1)\right), \quad i = 1, \dots, m.
  \]
  Then for $i = 1, \dots, m$, we have the following joint convergence in distribution,
  \[
  \left(\sum_{j'=1}^n Y_{ij'},B^{(n)}_i\right) \overset{d}\longrightarrow
  (N_i,\eta_i(t)) \quad \text{ as } n \to \infty,
  \]
  where $N_i \sim \mathrm{Poisson}(1)$ and $\eta_i(t) \sim \mathrm{Poisson}(1-e^{-t})$, with $N_i$ independent of $\eta_i(t))$ and all random variables independent over $i$.
\end{lem}

\begin{proof}
  By the law of rare events, for each $i=1, \dots, m$,
  \[
  \sum_{j'=1}^n Y_{ij'} \overset{d}\longrightarrow N_i \text{ as } n \to \infty.
  \]
  Notice that
  \[
n \cdot p^{(n)}_t(-1,1) = 1 - e^{-t} ,
  \]
  so that, conditional on $\sum_{j'=1}^n Y_{ij'}$, by the law of rare events,
  \[
B^{(n)}_i \overset{d}\longrightarrow \eta_i(t) \text{ as } n \to \infty,
  \]
  independent of $N_i$.
\end{proof}

Now we have the necessary elements for the proof of the main theorem in the Poisson setting. 

\begin{proof}[Proof of Theorem \ref{thm:poisson_inequality}]
  Let $g$ be a function of the form given in \eqref{eq:g_form}.
  By Theorem  \ref{thm:biased_hypercube_pisier}, we have that
  \begin{equation}\label{eq:cor2.1prelimit}
    \left(\mathbb E\|g(\varepsilon) - \mathbb E g(\varepsilon)\|^p\right)^{\frac{1}{p}}
    \leq 4 \left(1 - \frac{1}{n}\right)
    \int_0^\infty  
    \left(\mathbb E\!\left\|\frac{1}{n}\sum_{i=1}^m \sum_{j=1}^n\delta_{ij}(t)D_{ij} g(\varepsilon)\right\|^p\right)^{\frac{1}{p}} dt.
  \end{equation}
  We decompose the sum based on the value of $\varepsilon_{ij}$ as follows
  \begin{align}
  \frac{1}{n} \sum_{i=1}^m \sum_{j=1}^n\delta_{ij}(t)D_{ij} g(\varepsilon)
  &= \frac{1}{n}\sum_{i=1}^m \sum_{j=1}^n \mathds{1}_{\{\varepsilon_{ij} = 1\}} \delta_{ij}(t)D_{ij} g(\varepsilon)  \notag\\
  & \qquad +
  \frac{1}{n}\sum_{i=1}^m \sum_{j=1}^n \mathds{1}_{\{\varepsilon_{ij} = -1\}} \delta_{ij}(t)D_{ij} g(\varepsilon). \label{eq:sum_decomposision}
  \end{align}



  Also, for all $\varepsilon$,
  \[
  \|D_{ij} g(\varepsilon)\| \leq 2\|g\|_\infty =  2\|f\|_\infty,
  \]
  which is finite by assumption.
  Therefore, by the triangle inequality,
  \begin{align*}
    \lim_{n\to\infty}
   \left\|  \frac{1}{n} \sum_{i=1}^m \sum_{j=1}^n \mathds{1}_{\{\varepsilon_{ij} = 1\}} \delta_{ij}(t)D_{ij} g(\varepsilon) 
   \right\| \leq  2\|f\|_\infty \lim_{n\to\infty}
    \frac{1}{n} \sum_{i=1}^m \sum_{j=1}^n \mathds{1}_{\{\varepsilon_{ij} = 1\}} |\delta_{ij}(t)| = 0, 
  \end{align*}
  where by using the explicit expression of $\delta_{ij}(t)$ conditioned on $\varepsilon_{ij} = 1$, we see that 
  \[
  \mathds{1}_{\{\varepsilon_{ij} = 1\}} |\delta_{ij}(t)| \leq \max \left\{\left|\frac{e^{-t}}{\frac{2}{n}(1-e^{-t})+ 2e^{-t}}\right|, \left|\frac{e^{-t}}{\frac{2}{n}(1-e^{-t})-2 + 2e^{-t}} \right| \right\}
  \]
  is bounded as $n \to \infty$.
  Now we consider the second sum in \eqref{eq:sum_decomposision}:
  \begin{align*}
    \frac{1}{n}&\sum_{i=1}^m \sum_{j=1}^n \mathds{1}_{\{\varepsilon_{ij} = -1\}} \delta_{ij}(t)D_{ij} g(\varepsilon) \\
    &= \frac{1}{n} \sum_{i=1}^m \frac{1}{2}
    \left(
    f\!\left(\dots, \sum_{j'=1}^n Y_{ij'}, \, \dots \right) \!-
    f\!\left(\dots, \sum_{j'=1}^n Y_{ij'} +1, \, \dots \right) \right)
    \! \sum_{j=1}^n \mathds{1}_{\{\varepsilon_{ij} = -1\}}\delta_{ij}(t).
  \end{align*}
  Let $B^{(n)}_i$ be as in Lemma \ref{lem:binomial_limit}.
  Then, for any $i=1,\dots,n$, we obtain
  \[
  \frac{1}{n}\sum_{j=1}^n \mathds{1}_{\{\varepsilon_{ij} = -1\}}\delta_{ij}(t) =
  - B^{(n)}_i \frac{e^{-t}}{2(1-e^{-t})} + \frac{n-\sum_{j'=1}^n Y_{ij'}-B^{(n)}_i}{n} \cdot \frac{e^{-t}}{2- \frac{2}{n}(1-e^{-t})}.
  \]
  Therefore, by Lemma \ref{lem:binomial_limit}, we have that
  \begin{align*}
  - B^{(n)}_i \frac{e^{-t}}{2(1-e^{-t})} &\overset{d}\longrightarrow -\frac{e^{-t}}{2(1-e^{-t})}\eta_i(t) \text{ as } n \to \infty,
  \end{align*}
  and
  \begin{align*}
  \frac{n-\sum_{j'=1}^n Y_{ij'}-B^{(n)}_i}{n} \cdot \frac{e^{-t}}{2- \frac{2}{n}(1-e^{-t})}  &\overset{d}\longrightarrow \frac{e^{-t}}{2} \text{ as } n \to \infty.
  \end{align*}
  Finally, we conclude that
  \[
    \frac{1}{n}\sum_{i=1}^m \sum_{j=1}^n \mathds{1}_{\{\varepsilon_{ij} = -1\}} \delta_{ij}(t)D_{ij} g(\varepsilon)
    \]
    converges in distribution to
    \[
    \frac{1}{4} \sum_{i=1}^m (f(N_1, \dots, N_m) - f(N_1, \dots,N_i +1, \dots, N_m))
    \left(e^{-t}- \frac{e^{-t}}{(1-e^{-t})}\eta_i(t)\right).
  \]
%
  Thus, using the assumption that $f$ is bounded, we achieve the desired result by taking $n\to\infty$ in \eqref{eq:cor2.1prelimit}.
  \end{proof}

  \begin{proof}[Proof of Corollary \ref{cor:poisson_type_ineq}]
    By a similar symmetrization argument to that in the proof of Corollary \ref{cor:biased_hypercube_type}, observing that $\mathbb E \tilde \eta_i(t) = 0$, we have that 
    \[
      \mathbb E\!\left\|\sum_{i=1}^m \tilde\eta_i(t) D_i^{\mathbb{Z}} f(N)\right\|^p \leq 2^p T_p(X)^p\sum_{i=1}^m\mathbb E \!
      \left\|\tilde\eta_i(t)D_i^{\mathbb{Z}} f(N)\right\|^p.
    \]
    By the fact that the $\tilde\eta_i(t)$ are i.i.d.\  and independent of $D_i^{\mathbb{Z}} f(N)$, we have by Theorem \ref{thm:poisson_inequality} that 
    \[
\left(\mathbb E\|f(N) - \mathbb E f(N)\|^p\right)^{\frac{1}{p}} \leq 2
T_p(X) 
\left( \int_0^\infty   \left(\mathbb E \!
\left|\tilde\eta_1(t)\right|^p\right)^\frac{1}{p} \, dt\right)
\left(\sum_{i=1}^m\mathbb E \!
\left\|D_i^{\mathbb{Z}} f(N)\right\|^p\right)^\frac{1}{p}.
  \]
  Notice that 
  \[
    \left(\mathbb E \!
    \left|\tilde\eta_1(t)\right|^p\right)^\frac{1}{p} \leq \left(\mathbb E \!
    \left|\tilde\eta_1(t)\right|^2\right)^\frac{1}{2} = \frac{e^{-t}}{\sqrt{1-e^{-t}}}
  \]
  and 
  \[
    \int_0^\infty \frac{e^{-t}}{\sqrt{1-e^{-t}}} \, dt = 2,
  \]
  which concludes our proof.
  \end{proof}

\bibliographystyle{plain}
\bibliography{references}

\begin{thebibliography}{10}

\bibitem{bobkov1997poincare}
S.~Bobkov and M.~Ledoux.
\newblock Poincar\'e's inequalities and {T}alagrand's concentration phenomenon for the exponential distribution.
\newblock {\em Probab. Theory Related Fields}, 107(3):383--400, 1997.

\bibitem{CorderoErausquin2023}
Dario Cordero-Erausquin and Alexandros Eskenazis.
\newblock Talagrand's influence inequality revisited.
\newblock {\em Anal. PDE}, 16(2):571--612, 2023.

\bibitem{CorderoErausquin2024}
Dario Cordero-Erausquin and Alexandros Eskenazis.
\newblock Discrete logarithmic {S}obolev inequalities in {B}anach spaces.
\newblock {\em J. Lond. Math. Soc. (2)}, 109(2):Paper No. e12873, 24, 2024.

\bibitem{diaconis1996logarithmic}
P.~Diaconis and L.~Saloff-Coste.
\newblock Logarithmic {S}obolev inequalities for finite {M}arkov chains.
\newblock {\em Ann. Appl. Probab.}, 6(3):695--750, 1996.

\bibitem{enflo1978infinite}
Per Enflo.
\newblock On infinite-dimensional topological groups.
\newblock {\em S{\'e}minaire Maurey-Schwartz}, pages 1--11, 1978.

\bibitem{eskenazis2023some}
Alexandros Eskenazis.
\newblock Some geometric applications of the discrete heat flow.
\newblock {\em arXiv preprint arXiv:2310.01868}, 2023.

\bibitem{hytonen2016analysis}
Tuomas Hyt\"{o}nen, Jan van Neerven, Mark Veraar, and Lutz Weis.
\newblock {\em Analysis in {B}anach spaces. {V}ol. {I}. {M}artingales and {L}ittlewood-{P}aley theory}, volume~63.
\newblock Springer, Cham, 2016.

\bibitem{Ivanisvili2020}
Paata Ivanisvili, Ramon van Handel, and Alexander Volberg.
\newblock {Rademacher type and Enflo type coincide}.
\newblock {\em Annals of Mathematics}, 192(2):665 -- 678, 2020.

\bibitem{lee2017separators}
James~R. Lee.
\newblock Separators in region intersection graphs, 2017.

\bibitem{naor2012introduction}
Assaf Naor.
\newblock An introduction to the {R}ibe program.
\newblock {\em Japanese Journal of Mathematics}, 7:167--233, 2012.

\bibitem{naor2021average}
Assaf Naor.
\newblock An average {J}ohn theorem.
\newblock {\em Geometry \& Topology}, 25(4):1631--1717, 2021.

\bibitem{pisier1986}
Gilles Pisier.
\newblock Probabilistic methods in the geometry of {B}anach spaces.
\newblock In Giorgio Letta and Maurizio Pratelli, editors, {\em Probability and Analysis}, pages 167--241, Berlin, Heidelberg, 1986. Springer Berlin Heidelberg.

\bibitem{rabinovich2008average}
Yuri Rabinovich.
\newblock On average distortion of embedding metrics into the line.
\newblock {\em Discrete \& Computational Geometry}, 39(4):720--733, 2008.

\bibitem{talagrand1993isoperimetry}
Michel Talagrand.
\newblock Isoperimetry, logarithmic {S}obolev inequalities on the discrete cube, and {M}argulis' graph connectivity theorem.
\newblock {\em Geometric \& Functional Analysis GAFA}, 3(3):295--314, 1993.

\bibitem{talagrand1994on}
Michel Talagrand.
\newblock On {R}usso's approximate zero-one law.
\newblock {\em The Annals of Probability}, 22(3):1576--1587, 1994.

\bibitem{van2014probability}
Ramon van Handel.
\newblock {\em Probability in high dimension}.
\newblock Lecture notes, 2016.

\end{thebibliography}

{
\footnotesize
{\scshape 
\noindent
Program in Applied \& Computational Mathematics, Princeton University, Princeton, NJ 08544, USA.}

{
  \noindent
  \textit{\bfseries Email:} \href{mailto:miriam_gordin@brown.edu}{\texttt{miriam\_gordin@brown.edu}}
}
}

\end{document}